\documentclass[11pt,twoside]{amsart}
\usepackage{amssymb}
\usepackage{amsmath}
\usepackage{graphicx,color}

\newtheorem{thm}{Theorem}[section]

\newtheorem{lem}[thm]{Lemma}

\newtheorem{rem}[thm]{Remark}

\newcommand{\skipit}[1]{{}}
\newcommand{\prfend}{\hbox to7pt{\hfil}
\par\vskip-\baselineskip\hbox to\hsize
{\hfil\vbox {\hrule width6pt height6pt}}\vskip\baselineskip}

\newcommand {\PP}{\mathbb{P}}

\newcommand{\cL}{\mathcal{L}}

\newcommand{\K}{\mathbb{K}}

\newcommand{\myarrow}[2]{\hbox to #1pt{\hfil$\to$\hfil}{\hskip-#1pt{\raise
10pt\hbox to#1pt{\hfil$\scriptscriptstyle #2$\hfil}}}}

\begin{document}

\title{Lefschetz property and powers of linear forms in $\K[x,y,z]$}
\author[Charles Almeida]{Charles Almeida}
\address{Instituto de Matem\'atica, Estat\'istica e Computa\c{c}\~{a}o Cient\'ifica - UNICAMP, Rua S\'ergio Buarque de Holanda 651, Distr. Bar\~ao Geraldo, CEP 13083-859, Campinas (SP), Brasil}
\email{charlesalmd@gmail.com}

\author[Aline V. Andrade]{Aline V. Andrade}
\address{Instituto de Matem\'atica, Estat\'istica e Computa\c{c}\~{a}o Cient\'ifica - UNICAMP, Rua S\'ergio Buarque de Holanda 651, Distr. Bar\~ao Geraldo, CEP 13083-859, Campinas (SP), Brasil}
\email{aline.vandrade88@gmail.com}

\begin{abstract} In \cite{MMN-2012}, Migliore, Mir\'o-Roig and Nagel, proved that if $R = \K[x,y,z]$, where $\K$ is a field of characteristic zero, and $I=(L_1^{a_1},\dots,L_r^{a_4})$ is an ideal generated by powers of 4 general linear forms, then the multiplication by the square $L^2$  of a general linear form $L$  induces an homomorphism of  maximal rank in any graded component of $R/I$. More recently, Migliore and Mir\'o-Roig proved in \cite{MM-2016} that the  same is true for any number of general linear forms, as long the  powers are uniform. In addition, they conjecture that the same holds for arbitrary powers. In this paper we will solve this conjecture and  we will prove that if $I=(L_1^{a_1},\dots,L_r^{a_r})$ is an ideal of $R$ generated by arbitrary powers of any set of general  linear forms, then the multiplication by the square $L^2$ of a general linear form $L$ induces an homomorphism of  maximal rank in any graded component of $R/I$.

\end{abstract}

\thanks{   Almeida  was supported by FAPESP process numbers 2016/14376-0 and 2014/08306-4 and Andrade  was  supported by CAPES process number 99999.000282/2016-02.
\\ {\it 2010 Mathematic Subject Classification.} 14C20, 13D40}

\maketitle

\tableofcontents

\markboth{Charles Almeida, Aline V. Andrade}{Lefschetz property and powers of linear forms in $\K[x,y,z]$}

\today

\large

\section{Introduction}

Let $A = R/I$ be a artinian standard graded algebra, with $R = \K[x_1,\cdots,x_n]$ where $\K$ is a  field of characteristic zero. It is an important question to determine whether A has the {\it Strong Lefschetz Property} (SLP), that is, when the homomorphism induced by the multiplication map $\times L^k: A_{j-k} \rightarrow A_j$, of a general linear form $L$, has maximal rank in all degrees, or the {\it Weak Lefschetz Property} (WLP) that is, when the  multiplication map $\times L: A_{j-1} \rightarrow A_j$, of a general linear form $L$, has maximal rank in all degrees.
At first glance this might seem to be a simple problem in linear algebra, but instead it has proven to be extremely hard even in the case of very natural families of artinian graded algebras. There is a huge literature in this subject and the problems has been solved from different points of view applying tools of representation theory, vector bundles, differential geometry, among others (See \cite{MeM-2016}, \cite{MMO-2013} and \cite{MMi-2016}).

In this paper we deal with ideals $I \subset \mathbb{K}[x,y,z]$ generated by powers of linear forms which provides a great number of examples in which $A$ has $SLP$ (see for instance \cite{Stan} and \cite{Wat}) or $WLP$ (see for instance \cite{Sc-Se2010}). But, there is also a great number of examples in which $A$ fails to have $WLP$ and $SLP$ (see for instance \cite{MMN-2012}). In this direction, H. Schenck and A. Seceleanu proved that if $n = 3$, and $I$ an  ideal generated by powers of general linear forms, the algebra $A$ has $WLP$. Later, Migliore, Mir\'o-Roig and Nagel, proved that in the same ring, if $I \subset \mathbb{K}[x,y,z]$ is generated by powers of at most $4$ linear forms, then the multiplication by $L^2$, where $L$ is an general linear form has maximal degree in all rank (See \cite{MMN-2012}).

In a recent paper, Migliore and Mir\'o-Roig, proved that if $I \subset \mathbb{K}[x,y,z]$ is an ideal generated by uniform powers of any number of  linear forms, then  the multiplication by $L^2$ where $L$ is a general linear form, has maximal rank in all degrees. In such paper, they conjecture that the result is true even if the powers are not uniform (\cite{MM-2016}, Conjecture 4.5). Our goal is to show this conjecture, that is to prove

\begin{thm}\footnote{It came to our knowledge that this theorem was also proved independently by U. Nagel and J. Migliore.} For any artinian quotient of $\K[x,y,z]$ generated by powers of general linear forms, and for a general linear form $L$, the multiplication by $L^2$ has maximal rank in all degrees.
\end{thm} 

\noindent \underline{Acknowledgments:} This paper was written while we were under supervision of Professor Rosa Maria Mir\'o-Roig  at IMUB University of Barcelona. We would like to thank Professor Rosa Maria Mir\'o-Roig for the close support that she provided and for the several suggestions that helped improve this work. As well we would like to thank her and IMUB for the warm hospitality during our visit. We also would like to thank Darcy Camargo for his valued
help in the proof of inequality (\ref{inequality}).

\section{Preliminaries}

In this section we fix the notation and state the results that we will need to prove the main result of this short note (Theorem 1.1). In this paper, we define $R = \K[x,y,z]$ where $\K$ is a field of characteristic zero. If $a \in \mathbb{R}$, we will denote by $[a]_{+} = \max\{0,a\}$, and use the convention that a binomial ${a \choose b}$ is zero if $a < b$.

For any Artinian ideal $I \subset R$ and any general linear form $L \in R$, we have the exact sequence of $\K$-vector spaces:
\[
\cdots \longrightarrow  [R/I]_{m-2} \stackrel{\times L^2}{\longrightarrow}  [R/I]_{m} \longrightarrow [R/(I,L^2)]_{m} \longrightarrow 0.
\]
Therefore, the morphism  $\times L^2$ has maximal rank in degree $m$ if, and only if:

\[ \dim_{\K}[R/(I,L^2)]_m = [dim_{\K}[R/I]_m - dim_{\K}[R/I]_{m-2}]_+ \]

To compute such dimensions,  we will strongly use the following result from Ensalem and Iarrobino (\cite{En-Ia1995}; Theorem 2):
\begin{thm}\label{Duality}

Let $( L_1^{a_1}, \cdots, L_n^{a_r}) \subset R$ be an ideal generated by powers of $n$ linear forms. Let $\mathfrak{p}_1, \cdots,\mathfrak{p}_r $ be the ideals of $r$ points in $\mathbb{P}^{n-1}$(Each point is actually obtained explicitly from the corresponding linear form by duality). Choose positive integers $a_1,\cdots, a_r$. Then, for any integer $j \geq$ $ \max\{a_i\}$, one has
$$[\dim_{\K}[R/( L_1^{a_1}, \cdots, L_r^{a_r})]_{j}  = \dim_{\K}[\mathfrak{p}_1^{j-a_1+1} \cap  \cdots \cap \mathfrak{p}_r^{j-a_r+1}]_j .$$
\end{thm}

We will denote the linear system $[\mathfrak{p}_1^{a_1} \cap  \cdots \cap \mathfrak{p}_r^{a_n}]_j \subset [R]_j$ by $\mathcal{L}_{2}(j; a_1, \cdots, a_n)$ and; we will consider it as vector space and not as projective space when computing its dimension. Furthermore, we will use superscript to indicate repeated entries. For instance, $\mathcal{L}_2(j;2^3,10^2) = \mathcal{L}_2(j;2,2,2,10,10)$. 

It is well known that for any linear system one has that

\[\dim_{K}\mathcal{L}_{2}(j, a_1, \cdots, a_n) \geq \left[{j+2 \choose 2} - \sum_{i = 1}^{n} {a_i+1 \choose 2}\right]_+.\]

When the inequality is strict, we say that the linear system is {\em special}, otherwise, we say that the linear system is {\em non special}. It is a hard problem in Algebraic Geometry to  determine whether a linear system is special or not.
A linear system $\mathcal{L}_{2}(j; a_1, \cdots, a_n)$ is said to be in {\em standard position} if $j \geq a_1 +a_2  + a_3$.
   In \cite{DL-2007}, De Volder and Laface   showed that any  standard linear system $\mathcal{L}_{2}(j; a_1, \cdots, a_n)$ is non special. In this paper, using Cremona transformations, we will often we able to pass from a linear system  $\mathcal{L}_{2}(j; a_1, \cdots, a_n)$ to a linear system in standard position and then compute its dimension. Indeed, we have the following useful result for  our computations (see  \cite{Nagata}, \cite{LU}, or \cite{Dumnicky}, Theorem 3).
\begin{lem}
  \label{Lem1}
Let $n >   2$ and let $j, b_1,\cdots,b_n$ be non-negative integers with $b_1 \geq \dots \geq b_n$.  Set $m =  j - (b_1 + b_2 + b_{3})$. If $b_i + m \geq 0$ for all $i = 1,2,3$, then
\[
\dim_{\K}\mathcal{L}_2 (j; b_1,\ldots,b_n) = \dim_k \mathcal{L}_2 (j + m; b_1 +m,b_2+m,b_{3} +m , b_{4},\ldots,b_n).
\]
\end{lem}

We end these preliminaries with a useful application of Bezout's theorem.

\begin{rem} \label{bezout} \rm
 Assume the points $P_1,\dots,P_n$ are general.  If $2j < b_1 + \dots +b_5$ then
\[
\dim _{\K}\cL_2(j;b_1,\dots,b_n) = \dim _{\K} \cL_2 (j-2; b_1-1,\dots,b_5-1,b_6,\dots, b_n).
\]
If $j < b_1 + b_2$ then
\[
\dim _{\K }\cL_2(j;b_1,\dots,b_n) = \dim _{\K }\cL_2 (j-1; b_1-1,b_2-1,b_3,\dots, b_n).
\]
\end{rem}


\section{Main Result}

For ideals generated by powers of linear forms in $\K[x,y,z]$, the following two facts are known:

\vskip 2mm
\begin{itemize}
\item[(1)] An artinian ideal in $\K[x,y,z]$ generated by powers of arbitrary linear forms has the WLP (see \cite{Sc-Se2010}, Main Theorem).
\item[(2)] Let $I=(L_1^{a_1},\dots,L_r^{a_r})\subset R=\K[x,y,z]$ be an artinian ideal generated by powers of general linear forms and let $L$ be a general linear form. The multiplication $\times L^j : [R/I]_{t-j} \rightarrow [R/I]_t$ by $L^j$ for $j\ge 3$ does not necessarily have maximal rank (see \cite{MM-2016}, Theorems 5.1, 5.2 and 5.3).
\end{itemize}

This leaves open whether the multiplication by the square $L^2$ of a general linear form $L$ has maximal rank. Again for ideals generated by powers of general linear forms two results are known: the case of almost complete intersections and the case of uniform powers. Indeed, it holds:

\vskip 2mm
\begin{itemize}
\item[(1)] Let $L_1,\dots,L_4, L$ be five general linear forms in $R = \K[x,y,z]$. Set $I=(L_1^{a_1},\dots,L_4^{a_4})$ and $A = R/I$. Then, for each integer $j$, the multiplication map $\times L^2 : [A]_{j-2} \rightarrow [A]_j$ has maximal rank (see \cite{MMN-2012}, Proposition 4.7).
\item[(2)]  Let $L_1,\dots,L_r,L\in \K[x,y,z]$ be $r+1$ general linear forms. Let $I$ be the ideal $(L_1^{k},\dots,L_r^{k})$. Then, for each integer $j$, the multiplication map
 $\times L^2:[R/I]_{j-2}\longrightarrow [R/I]_{j}$
 has maximal rank (see \cite{MM-2016}, Theorem 4.4).
\end{itemize}

This section is entirely devoted to prove that if $I$ is generated by any powers of any number of general linear forms and $L$ is any general linear form then the multiplication by $L^2$  has maximal rank in any degree. This result solves a conjecture stated by Migliore and Mir\'{o}´-Roig (see \cite{MM-2016}, Conjecture 4.5). In fact, we have:

\begin{thm}\label{mainteo} For any artinian quotient of $\K[x,y,z]$ generated by powers of general linear forms, and for a general linear form $L$, the multiplication by $L^2$ has maximal rank in all degrees.
\end{thm}

\begin{proof}
Let $I = ( l_1^{a_1},\cdots, l_r^{a_r} )$ where $l_i$ are general linear forms for $i = 1, \cdots, r$ in $\K [x,y,z]$. Note that we can assume that $r \geq 5$, because as we have just pointed out the case $r\leq 4$ was solved in \cite{MMN-2012}, Proposition 4.7, and, we will also assume that $(a_1, \cdots, a_r) \neq (k,\cdots,k)$ for any $k$, because otherwise this is proved in \cite{MM-2016}.
Set $A = R/I$.  Without loss of generality we can also suppose that $$0 < a_1 \leq a_2 \leq \cdots \leq a_r.$$

We split the proof in two cases:
\begin{itemize}
\item[(i)] For all $m \geq 2$ we have that

$$
a_{m+1} \leq \dfrac{\sum_{i=1}^{m}a_i -m}{m-1}.
$$
\item[(ii)] There exists 
$m \geq 2$ such that

$$
a_{m+1} > \dfrac{\sum_{i=1}^{m}a_i -m}{m-1}.
$$
\end{itemize}

\vskip 2mm
\noindent \underline{Case (i)} Assume that  for all $m \geq 2$ we have 

\begin{equation}\label{fundamental}
a_{m+1} \leq \dfrac{\sum_{i=1}^{m}a_i -m}{m-1}.
\end{equation}
By \cite{Sc-Se2010}, any ideal $J$ generated by powers of general linear forms has WLP and its Hilbert function is unimodal, i.e. there is a unique  integer $p$ such that $$\dim_{\K}[R/J]_{p-1}\le \dim_{\K}[R/J]_{p} \ge \dim_{\K}[R/J]_{p+1}.$$
By \cite{Sc-Se2010} Lemma 2.1 and Lemma 2.2 $A$ has WLP, so the Hilbert function of $A$ is unimodal, and even more the hypothesis (i) guarantees that $A/LA$ is minimally generated and  the socle degree of $A/LA$ is $p =\displaystyle \lfloor \frac{\sum_{i=1}^{r}a_i -r}{r-1} \rfloor $. Hence one has the following chain

$$[A]_{0} \hookrightarrow [A]_1 \hookrightarrow \cdots [A]_{p-1} \hookrightarrow [A]_{p}  \twoheadrightarrow  [A]_{p+1} \twoheadrightarrow  [A]_{p+2} \twoheadrightarrow\cdots $$

\noindent of injections and surjections. This observation narrows down our study of the multiplication map $\times L^2$ and 
 it will be  enough to check if 
 $$\times L^2: [A]_{p-1} \to [A]_{p+1},$$
 \noindent with $p :=\displaystyle \lfloor \frac{\sum_{i=1}^{r}a_i -r}{r-1} \rfloor $ has maximal rank. To see this, we are going to show that :
\begin{equation}\label{main}
\dim_{\K} [A/L^2A]_{p+1} = [\dim_{\K} [A]_{p+1} - \dim_{\K} [A]_{p-1}]_+
\end{equation}

Write $\displaystyle \sum_{i=1}^{r}a_i  = (r-1)(p+1) + b$ with $1 \leq b \leq r-1$. Note that if $b = 1$, then $p =\dfrac{\sum_{i=1}^{r}a_i - r}{r-1} $.

First, let us compute the left hand side of (\ref{main}). By Theorem \ref{Duality}, one has that

\begin{equation}\label{leftdim}
\begin{array}{rcl} \dim_{\K} [A/L^2A]_{p+1} & = & \dim_{\K}[R/( L^2, I )]_{p+1} \\
&  = & \dim_{\K}[\mathfrak{p} \cap \mathfrak{p}_1 \cap  \cdots \cap \mathfrak{p}_r]_{j} \\
& = & 
 \dim_{\K} \mathcal{L}_{2}(p+1;p, p-a_1 +2, \cdots ,p-a_r+2).
\end{array} \end{equation}

To compute the dimension of the linear system $\mathcal{L}_{2}(p+1;p, p-a_1 +2, \cdots ,p-a_r+2)$ we consider the lines $L_i$ passing through the points $\mathfrak{p}$ and $\mathfrak{p}_i$. By B\'ezout's Theorem (see remark \ref{bezout}), the line $L_i$ appears with multiplicity at least $B_i$ in the base locus of the linear system $\mathcal{L}_{2}(p+1;p, p-a_1 +2, \cdots ,p-a_r+2)$, where $B_i$ is given by

$$B_i = [p+p-a_i+2 - p -1]_+ = [p-a_i +1]_+ = p-a_i + 1$$

\noindent and the last equality follows from the fact that $a_i \leq p+1$ for all $i=1, \cdots ,r$. Hence, we have

\begin{equation}\label{leftdimdone} \begin{array}{rcl} \dim_{\K} [A/L^2A]_{p+1} & = & \dim_{\K} \mathcal{L}_{2}(p+1;p, p-a_1 +2, \cdots ,p-a_r+2) \\
& = & \dim_{\K}\mathcal{L}_{2}(b;b-1,1^r) \\ 
&  = &  \left[ {b+2\choose2} - {b\choose2} -r \right]_+ \\
&  = & [2b +1 -r]_+ .
\end{array}
\end{equation}

\vspace{0,5cm}
Now, we are going to compute the right-hand side of (\ref{main}), namely, $ [ \dim_{\K} A_{p+1} -\dim_{\K} A_{p-1}]_+ $.

We have $a_i \leq p+1$ for all $i = 1, \cdots, r$. Therefore, if there exists $j$ such that $a_{j+1} \geq p$, we will have

$$\dim_{\K}[ A]_{p-1} = \dim [R/( l_1^{a_1}, \cdots, l_r^{a_r})]_{p-1} = \dim [R/(l_1^{a_1}, \cdots, l_j^{a_j})]_{p-1}. $$

Let $s$ be the number of $a_i$'s equals to $p$ and let $t$ be the number of $a_i$'s equals to $p+1$. Then, by Theorem \ref{Duality}, we have:

$$[ \dim_{\K} A_{p+1} -\dim_{\K} A_{p-1}]_+ = $$
$$[\dim_{\K} \mathcal{L}_{2}(p+1;p-a_1 +2, \cdots p-a_j+2,2^s,1^t) - \dim_{\K} \mathcal{L}_{2}(p-1;p-a_1, \cdots, p-a_j)]_+.$$

Note that if  $a_i \leq p-1$ for all $i=1, \cdots , r$,  then $s = t = 0$ and $ j=r$ in the above equality. Furthermore, the linear system $\mathcal{L}_{2}(p+1;p-a_1 +2, \cdots, p-a_j+2,2^s,1^t)$ is in standard form if, and only if, $a_1 + a_2 + a_3 \geq 2p+5$ and the linear system $\mathcal{L}_{2}(p-1;p-a_1, \cdots, p-a_j)$ is in standard form if,  and only if, $a_1 + a_2 + a_3 \geq 2p+1$.

But, recall that  $\displaystyle \sum_{i=1}^{r}a_i  = (r-1)(p+1) + b$ with $1 \leq b \leq r-1$. So, we have:

$$a_1 +a_2 +a_3 = (r-1)(p+1) + b - \displaystyle \sum_{i=4}^{r}a_i. $$


Then for $r \geq 5$ and from the fact that $b \geq 1$, one has

\[\begin{array}{rcl} a_1 +a_2 +a_3 & \geq & (r-1)(p+1) + 1 - \displaystyle \sum_{i=4}^{r}a_i \\
&  \geq & (r-5)(p+1) + 4p +5  - \displaystyle \sum_{i=4}^{r}a_i \\
& \geq & (r-5)(p+1)  - \displaystyle \sum_{i=6}^{r}a_i  + 2p - a_4 - a_5 + 2p +5.
\end{array}
\]

Therefore, since $a_i \leq p+1$ for all $i = 1, \cdots,r$, one has that

$$a_1 +a_2 +a_3  - (2p+5)\geq 2p - a_4 - a_5 \geq -2 $$

\noindent or, equivalently

\[a_1 +a_2 +a_3  - (2p+1 ) \geq 2p - a_4 - a_5 +4 \geq 2.\]

This implies that the linear system $\mathcal{L}_{2}(p-1;p-a_1, \cdots ,p-a_j)$  is in standard form. Let us first assume that  $a_1 +a_2 +a_3  - (2p+5)\geq 0 $, then the linear system  $\mathcal{L}_{2}(p+1;p-a_1 +2, \cdots ,p-a_j+2,2^s,1^t)$ is also in standard form. Under this assumption we have

$$[\dim_{\K} \mathcal{L}_{2}(p+1;p-a_1 +2, \cdots, p-a_j+2,2^s,1^t) - \dim_{\K} \mathcal{L}_{2}(p-1;p-a_1, \cdots ,p-a_j)]_+ = $$

$$ \left[\left[\displaystyle {p+3 \choose 2} - \displaystyle \sum_{i=1}^{j} {p-a_i+3 \choose 2} - 3s -t\right]_{+} - \left[\displaystyle{p+1 \choose 2} - \displaystyle\sum_{i=1}^{j}{p -a_i +1 \choose 2}\right]_+\right]_{+}$$
 Observe that if 
 \begin{equation}\label{inequality}
 \displaystyle{p+1 \choose 2} - \displaystyle\sum_{i=1}^{j}{p -a_i +1 \choose 2} \geq 0.
 \end{equation}

\noindent then 

$$ \left[\left[\displaystyle {p+3 \choose 2} - \displaystyle \sum_{i=1}^{j} {p-a_i+3 \choose 2} - 3s -t\right]_{+} - \left[\displaystyle{p+1 \choose 2} - \displaystyle\sum_{i=1}^{j}{p -a_i +1 \choose 2}\right]_+\right]_{+} = $$

$$ = \left[\displaystyle {p+3 \choose 2} - \displaystyle \sum_{i=1}^{j} {p-a_i+3 \choose 2} - 3s -t - \displaystyle{p+1 \choose 2} - \displaystyle\sum_{i=1}^{j}{p -a_i +1 \choose 2}\right]_+ .$$

\noindent Note that 
 $$\displaystyle \sum_{i=1}^{r}{p -a_i +1 \choose 2} = r {p+1 \choose 2}- \dfrac{1}{2}(2p+1)((r-1)(p+1)+b) + \dfrac{1}{2} \displaystyle \sum_{i =1}^{r}a_i^2.$$

 Therefore, we have

 \[\displaystyle\sum_{i=1}^{r}{p -a_i +1 \choose 2}- \displaystyle{p+1 \choose 2}  \leq 0\]

 \[\Leftrightarrow (r-1) {p+1 \choose 2}- \dfrac{1}{2}(2p+1)((r-1)(p+1)+b) + \dfrac{1}{2} \displaystyle \sum_{i =1}^{r}a_i^2 \leq 0\]

 \begin{equation}\label{inequality2}
 \Leftrightarrow \sum_{i =1}^{r}a_i^2 \leq (r-1)(p+1)^2 + b(2p+1).
 \end{equation}
 
 Among all sequences $(a_1, \cdots, a_r) \neq (k, \cdots, k)$ such that $a_i \leq p+1$, and $\sum_{i = 1}^{r} a_i = (r-1)(p+1) + b$ the sequence $(b, p+1, \cdots, p+1)$ maximize $\sum_{i=1}^{r}a_{i}^2$. So we have:

 $$\displaystyle \sum_{i =1}^{r}a_i^2 \leq (r-1)(p+1)^2 + b^2.$$

 Hence, to see the inequality (\ref{inequality2}) it is enough to see that:

 $$(r-1)(p+1)^2 + b^2 - (r-1)(p+1)^2 - b(2p+1) \leq 0 $$

 $$ \Leftrightarrow b^2 + b(2p+1)\leq 0$$
 $$ \Leftrightarrow b- (2p+1) \leq 0.$$

But from the fact that $\displaystyle \sum_{i=1}^{r}a_i = (r-1)(p+1) + b \leq r(p+1) $ we have that $b \leq p+1$, therefore $b- (2p+1) \leq 0.$

Then
\vskip 2mm
$[\dim_{\K} \mathcal{L}_{2}(p+1;p-a_1 +2, \cdots, p-a_j+2,2^s,1^t) - \dim_{\K} \mathcal{L}_{2}(p-1;p-a_1, \cdots, p-a_j)]_+ = $

$= \left[\displaystyle {p+3 \choose 2} - \displaystyle \sum_{i=1}^{j} {p-a_i+3 \choose 2} - 3s -t - \displaystyle{p+1 \choose 2} - \displaystyle\sum_{i=1}^{j}{p -a_i +1 \choose 2}\right]_{+} = $

$=\left[2p+3 -\displaystyle\sum _{i=1}^{j}(2(p-a_i)+3) - 3s -t\right]_+ = $

$=[2b+1-r]_+.$

Hence, if $a_1 +a_2 +a_3 - 2p - 5 \geq 0$ we have proved the equality
 $$\dim_{\K} [A/L^2A]_{p+1} = [2b +1 -r]_+= [\dim_{\K} [A]_{p+1} - \dim_{\K} [A]_{p-1}]_+ .$$ 

Now assume $a_1 + a_2 +a_3 -2p -5 = -1$. So, $2p - a_4 - a_5 =-1$ or $2p - a_4 - a_5 = -2 $. In the first case we have that $2p - a_4 - a_5 =-1$ which implies   $a_4 = p$ and  $a_i = p+1$ for all $i \geq 5$. In addition, one has  $b = 1$, which implies that the left-hand side of the equality (\ref{main}) is $0$. To compute the right-hand side of (\ref{main}) we use Lemma \ref{Lem1} and we get

\[ \dim_{\K} \mathcal{L}_{2}(p+1;p-a_1 +2, \cdots, p-a_3+2,2,1^{r-4}) =  \dim_{\K} \mathcal{L}_{2}(p;p-a_1+1, \cdots p-a_3+1,2,1^{r-4}). \]

The second linear system is in standard form.  So one has:

\[\dim_{\K} \mathcal{L}_{2}(p;p-a_1+1, \cdots, p-a_3+1,2,1^{r-4}) = \left[{p+2 \choose 2} - \displaystyle \sum_{i=1}^3 {p-a_i+2 \choose 2} -3 -(r-4) \right]_+ \]

\noindent and

\[\dim_{\K} \mathcal{L}_{2}(p-1;p-a_1, p-a_2, p-a_3) = \left[{p+1 \choose 2} - \displaystyle \sum_{i=1}^3 {p-a_i+1 \choose 2}\right]_+ .\]

Since 

\[ \left[\left[{p+2 \choose 2} - \displaystyle \sum_{i=1}^3 {p-a_i+2 \choose 2} -3 -(r-4) \right]_+ - \left[{p+1 \choose 2} - \displaystyle \sum_{i=1}^3 {p-a_i+1 \choose 2}\right]_+ \right]_+  \leq\]

\[\label{equals-1} \left[{p+2 \choose 2} - \displaystyle \sum_{i=1}^3 {p-a_i+2 \choose 2} -3 -(r-4)  - {p+1 \choose 2} - \displaystyle \sum_{i=1}^3 {p-a_i+1 \choose 2} \right]_+ \]

\[ = [-2b-1 +2p +4 -r]_+ = [3-r]_+ = 0.\]

If $2p - a_4 - a_5 = -2 $ then $a_i = p+1$ for all $i \geq 4$. In addition, one has that $b = 2$ which implies that the left-hand side of the equality (\ref{main}) is $0$. To compute the right-hand side of (\ref{main}) we use Lemma \ref{Lem1} and we get 

\[ \dim_{\K} \mathcal{L}_{2}(p+1;p-a_1 +2, \cdots p-a_3+2,1^{r-3}) =  \dim_{\K} \mathcal{L}_{2}(p;p-a_1+1, \cdots p-a_3+1,1^{r-3}). \]

Again the second linear system is in standard form, and therefore one has that 
\[\dim_{\K} \mathcal{L}_{2}(p;p-a_1+1, \cdots, p-a_3+1,1^{r-3}) = \left[{p+2 \choose 2} - \displaystyle \sum_{i=1}^3 {p-a_i+2 \choose 2}  -(r-3) \right]_+ \]

Furthemore, from this, we have that

\[[\dim_{\K} [A]_{p+1} - \dim_{\K} [A]_{p-1}]_+  = \] 

\[ = \left[ \left[{p+2 \choose 2} - \displaystyle \sum_{i=1}^3 {p-a_i+2 \choose 2}  -(r-3) \right]_+ - \left[{p+1 \choose 2} - \sum_{i=1}^{3}{p-a_i+1 \choose 2}\right]_+  \right]_+ \leq  \]

\[ = \left[ {p+2 \choose 2} - \displaystyle \sum_{i=1}^3 {p-a_i+2 \choose 2}  -(r-3)  - {p+1 \choose 2} - \sum_{i=1}^{3}{p-a_i+1 \choose 2}  \right]_+ =   \] 

\[= [5-r]_+=0.\]

Therefore both sides of (\ref{main}) coincide. Finally it remains to see that what happens for the case $a_1 + a_2 + a_3 - 2p -5 = -2$. In this case we have that $a_i = p+1$
 for all $i \geq 4$. Again we have that $b = 1$ and applying  Lemma \ref{Lem1} we obtain:

 \[ \dim_{\K} \mathcal{L}_{2}(p+1;p-a_1 +2, \cdots, p-a_3+2,1^{r-3}) =  \dim_{\K} \mathcal{L}_{2}(p-1;p-a_1, \cdots, p-a_3,1^{r-3}). \]

With the same argument as before, one sees that :

\[ [\dim_{\K} \mathcal{L}_{2}(p+1;p-a_1 +2, \cdots, p-a_j+2,1^{r-3}) - \dim_{\K} \mathcal{L}_{2}(p-1;p-a_1, p-a_2, p-a_3]_+ \]
\[  = [\dim_{\K} \mathcal{L}_{2}(p+1;p-a_1 +2, \cdots, p-a_j+2,1^{r-3}) - \dim_{\K} \mathcal{L}_{2}(p-1;p-a_1, p-a_2, p-a_3]_+ =\]
\[ = [3-r]_+ = 0\]

\noindent and both sides of (\ref{main}) coincide which concludes the proof of the theorem in case (i). That is if  $a_{m+1} \leq \frac{\sum_{i=1}^{m}a_i -m}{m-1}$ for all $m \geq 2$, then multiplication by $L^2$ has maximal rank in all degrees.

\vskip 4mm
\noindent \underline{Case (ii)}
 Now, suppose that (\ref{fundamental}) is not satisfied, and let $m \geq 2$ be the least integer $j \geq 2$ such that $ a_{j+1} > \frac{\sum_{i=1}^{j}a_i -j}{j-1}$. Consider the ideals $\mathfrak{a} = ( l_1^{a_1}, \cdots, l_{m+1}^{a_{m+1}})$ and $\mathfrak{b} =  ( l_1^{a_1}, \cdots, l_{m}^{a_{m}} )$. Let $\mathcal{A} = R/\mathfrak{a}$, and $\mathcal{B} = R/\mathfrak{b}$. By the  case (i), $\times L^2: [\mathcal{B}]_{j-2} \to [\mathcal{B}]_j$ has maximal rank in all degrees $j$. We define $q  := \displaystyle \lfloor \dfrac{\sum_{i=1}^{m}a_i -m}{m-1} \rfloor$ and we consider the commutative diagram:

\begin{equation}
\begin{array}{cccc}
[\mathcal{B}]_{j-2} &  \stackrel{\times L^2}{\longrightarrow} & [\mathcal{B}]_{j}  \\
\downarrow &  & \downarrow  \\
\left[\mathcal{A}\right]_{j-2} &  \stackrel{\times L^2}{\longrightarrow} & [\mathcal{A}]_{j}\\
\end{array}
\end{equation}

\vskip 2mm

We have the following possibilities:

\vskip 2mm
\begin{itemize}
\item[a)] For $j < a_{m+1}$, one has that $[\mathcal{B}]_{j-2} \simeq [\mathcal{A}]_{j-2}$ $[\mathcal{B}]_{j} \simeq [\mathcal{A}]_j$. Therefore $\times L^{2} : [\mathcal{A}]_{j-2} \to [\mathcal{A}]_{j}$ has maximal rank.

\vskip 2mm
\item[b)] If $j > a_{m+1}$, since $q \leq \frac{\sum_{i=1}^{m}a_i -m}{m-1} < a_{m+1} $, one has that $j \geq q+2$. Therefore, $\times L^{2} : [\mathcal{B}]_{j-2} \to [\mathcal{B}]_{j}$ is surjective. But if the top row map  in the above commutative diagram is surjective then the bottom row map is surjective as well. So, we conclude that  $\times L^{2} : [\mathcal{A}]_{j-2} \to [\mathcal{A}]_{j}$ has maximal rank.

\vskip 2mm
\item[c)] Assume $j = a_{m+1}$. Since $a_{m+1} \geq q + 1$, we have two possibilities:  $a_{m+1} \geq q+2$ or $a_{m+1} = q+1$. If $a_{m+1} \geq q+2$ we argue as in case (b) and we get that  $\times L^{2} : [\mathcal{A}]_{j-2} \to [\mathcal{A}]_{j}$ is surjective. If $a_{m+1} = q+1$ we have the following commutative diagram:

\begin{equation}
\begin{array}{cccccc}
[\mathcal{B}]_{q-1} & \stackrel{\times L^2}{\longrightarrow} & [\mathcal{B}]_{q+1}  \\
\pi _{q-1}{\downarrow } &  & \downarrow \pi _{q+1} \\
 \left[\mathcal{A}\right]_{q-1} & \stackrel{\times L^2}{\longrightarrow} & [\mathcal{A}]_{q+1} 
\end{array}
\end{equation}

Where $\pi_{i}$ is the natural projection for $ i = q-1,q+1$ and hence it is surjective. If the upper line is surjective, $\times L^{2} : [\mathcal{A}]_{j-2} \to [\mathcal{A}]_{j}$ is also surjective by the commutativity of the above square. Otherwise, if the top map is injective but not surjective, one has that $(l_1^{a_1}, \cdots, \l_{m}^{a_m}, L^2)_{q+1}$ is not all $R_{q+1}$. Since in characteristic zero, $(p+1)-$th powers of general linear forms generate $[R]_{q+1}$,  $l_{m+1}^{a_{q+1}}$ must be outside the image of $[\mathcal{B}]_{q-1}$ in $[\mathcal{B}]_{q+1}$, so the map $\times L^{2} : [\mathcal{A}]_{q-1} \to [\mathcal{A}]_{q+1}$ is injective since $\times L^{2} : [\mathcal{B}]_{q-1} \to [\mathcal{B}]_{q+1}$ is injective.
\end{itemize}

Since for all $ m+1 \leq t \leq r$, $a_{t} \geq a_{m+1} \geq q+1$, and the socle degree of $R/(l_1^{a_1}, \cdots, l_{t}^{a_t}, L)$, which is the peak of the Hilbert function of $R/(l_1^{a_1}, \cdots, l_{t}^{a_t})$, is smaller or equal than $q$, repeating the same argument as before, we have that $\times L^2: [A]_{j-2} \to [A]_j$ has maximal rank for all  integers $j$ which concludes the proof of the Theorem \ref{mainteo}.

\end{proof}

\end{document}